\documentclass[preprint,authoryear, 12pt]{elsarticle}
\usepackage{graphicx,latexsym,euscript,makeidx,color,bm}
\usepackage{amsmath,amsfonts,amssymb,amsthm,thmtools,mathtools,mathrsfs,enumerate}
\usepackage[colorlinks,linkcolor=blue,anchorcolor=green,citecolor=red]{hyperref}
\usepackage{geometry}
\def\5n{\negthinspace \negthinspace \negthinspace \negthinspace \negthinspace }
\def\4n{\negthinspace \negthinspace \negthinspace \negthinspace }
\def\3n{\negthinspace \negthinspace \negthinspace }
\def\2n{\negthinspace \negthinspace }
\def\1n{\negthinspace }
\def\dbE{\mathbb{E}}       
\def\dbF{\mathbb{F}} \def\sF{\mathscr{F}}      
\def\dbG{\mathbb{G}}         
\def\dbH{\mathbb{H}}   \def\cH{{\cal H}}

\def\dbN{\mathbb{N}}

\def\dbR{\mathbb{R}}       \def\IR{\mathbb{R}}

\def\Om{\Omega}

\def\ss{\smallskip}                
\def\ms{\medskip}                
               
\def\ds{\displaystyle}

        \def\q{\quad}                      
\def\ns{\noalign{\ss}}    \def\qq{\qquad}                    
    \def\hb{\hbox}                     
                   
         \def\rf{\eqref}                    
  \def\deq{\triangleq}               
            \def\({\Big (}
\def\les{\leqslant}                  \def\){\Big )}
       \def\esssup{\mathop{\rm esssup}}   \def\[{\Big[}
           \def\]{\Big]}

                 \def\o{\omega}

    \def\t{\tau}     \def\f{\varphi}  \def\i{\infty}   

\def\bde{\begin{definition}\label}    \def\ede{\end{definition}}
\def\be{\begin{equation}}
\def\bel{\begin{equation}\label}      \def\ee{\end{equation}}
\def\bt{\begin{theorem}\label}        \def\et{\end{theorem}}
\def\bc{\begin{corollary}\label}      \def\ec{\end{corollary}}
\def\bl{\begin{lemma}\label}          \def\el{\end{lemma}}
\def\bp{\begin{proposition}\label}    \def\ep{\end{proposition}}
\def\bas{\begin{assumption}\label}    \def\eas{\end{assumption}}
\def\br{\begin{remark}\label}         \def\er{\end{remark}}
\def\bex{\begin{example}\label}       \def\ex{\end{example}}
\def\ba{\begin{array}}                \def\ea{\end{array}}
\def\ben{\begin{enumerate}}           \def\een{\end{enumerate}}

\newtheorem{theorem}{Theorem}[section]
\newtheorem{definition}[theorem]{Definition}
\newtheorem{proposition}[theorem]{Proposition}
\newtheorem{corollary}[theorem]{Corollary}
\newtheorem{lemma}[theorem]{Lemma}
\newtheorem{remark}[theorem]{Remark}

\newtheorem{assumption}[theorem]{Assumption}
\newtheorem{example}[theorem]{Example}

\makeatletter
   
   \@addtoreset{equation}{section}
\makeatother

\allowdisplaybreaks[4]
\begin{document}
\journal{Statist. Probab. Lett. on 8th December 2021}
\begin{frontmatter}
\title{On the uniqueness result for the BSDE with continuous coefficient}
%\tnotetext[fund]{Supported by the National Natural Science Foundation of China(No. 11601509).}
\author[math]{Yufeng Shi}
%\ead{yfshi@sdu.edu.cn}
\author[math]{Zhi Yang\corref{correspondingauthor}}
\ead{yzyyss85@163.com}
\address[math]{Institute for Financial Studies and School of Mathematics, Shandong University, Jinan 250100, Shandong, China}
\cortext[correspondingauthor]{Corresponding author}
\begin{abstract}
In this paper, we study one-dimensional backward stochastic differential equation (BSDE, for short), whose coefficient $f$ is Lipschitz in $y$ but only continuous in $z$. In addition, if the terminal condition $\xi$ has bounded Malliavin derivative, we prove some uniqueness results for the BSDE with quadratic and linear growth in $z$, respectively.
\end{abstract}


\end{frontmatter}
\section{Introduction}
We study backward stochastic differential equations (BSDEs) of the following type:
\begin{equation}\label{bsde}
Y_{t}=\xi+\int_{t}^{T} f\left(s, Y_{s}, Z_{s}\right) d s-\int_{t}^{T} Z_{s} d W_{s}, \quad 0 \leqslant t \leqslant T,
\end{equation}
where $\left(W_{t}\right)_{0 \leqslant t \leqslant T}$ be a standard $d$-dimensional Brownian motion on a probability space $(\Omega, \mathscr{F}, \mathbb{P})$, with $\mathbb{F} = \left(\mathscr{F}_{t}\right)_{0 \leqslant t \leqslant T}$ is the augmented natural filtration generated by $W$. Fixed $T>0$ the time horizon of the study is called the terminal time. The terminal condition $\xi$ is a $\mathscr{F}_{T}$-measurable random variable. The coefficient (or generator) $f$ of the BSDE \eqref{bsde} is assumed to be a function from $[0, T] \times \Omega \times \mathbb{R} \times$ $\mathbb{R}^{d}$ to $\mathbb{R}$ that is measurable with respect to Prog$\otimes \mathcal{B}(\mathbb{R}) \otimes \mathcal{B}\left(\mathbb{R}^{d}\right)$, where Prog is the progressive sigma-algebra on $[0, T] \times \Omega$. The triple $(\xi, T, f)$ is the parameters of the BSDE \eqref{bsde}.

\ms

Nonlinear BSDEs were first introduced by \cite{ref5}. They proved that when $\xi$ is square integrable and $f$ is uniformly Lipschitz in $(y,z)$ there exists a unique adapted solution with square integrability properties. Since then, BSDEs have been studied with great interest, due to their wide range of applications in mathematical finance, stochastic optimal control, and partial differential equations. In particular, many efforts have been made to relax the assumptions on the parameters $\xi $ and $f$ of the BSDE \eqref{bsde}. The full list of contributions is too long to give and we will only quote results in our framework. A class of BSDEs with the generator of quadratic growth concerning the variable $z$, has received a lot of attention in recent years. For convenience, hereafter, by a {\it quadratic BSDEs}, we mean that in BSDE \eqref{bsde}, the map $z\mapsto f(t,y,z)$ grows no more than quadratically. Concerning the scalar case, and restricting to the bounded terminal condition $\xi$, \cite{ref25} proved the existence of a solution $(Y,Z)$ such that $Y$ is a bounded scalar process. Since then, a large number of papers have dealt with extensions and applications.

\ms

In this paper, we will focus on the uniqueness results for the BSDEs in the one-dimensional case. Our main purpose is to strengthen the uniqueness result for the quadratic BSDEs with generators $f(t,y,z)$ that are only continuous in $z$. This is different from the previous works, which assume that the generators $f$ are convex functions with respect to $z$ (see e.g. \cite{ref08}, \cite{ref62}) or other additional conditions (see e.g. \cite{ref66}, \cite{ref69}, \cite{ref81}). On the other hand, we require the terminal conditions $\xi$ that have bounded Malliavin derivative. Let us recall our strategy in detail. To obtain the main result of this paper, we first use the sup-convolution approximation techniques (see \autoref{BSDEfbijin}) to get the approximate generators with regularity properties. Then the solution $(Y^n,Z^n)$ of the BSDE \rf{BSDEjieduan} with parameters $(\xi,f_n)$ converges to the maximal solution $(\overline{Y}, \overline{Z})$ of the BSDE \eqref{bsde} by the monotone stability proposition (see \autoref{BSDEbijin}). In addition, if the terminal condition $\xi$ is Malliavin differentiable with a bounded Malliavin derivative, then $Z^n$ is bounded uniformly with respect to $n$.
% (The similar results can refer to \cite{ref23}, \cite{ref35}, \cite{ref81})
Thus, we can get that $\overline{Z}$ is bounded. Moreover, applying similar arguments (see \autoref{BSDEfbijin1}), the minimal solution $(\underline{Y},\underline{Z})$ of the BSDE \eqref{bsde} is also bounded (see \autoref{BSDEbijin1}). Finally, noting that the $\overline{Z}$ and $\underline{Z}$ are bounded, we localize the generator $f$ by $\bar{f}$, then under Assumptions (A1)-(A2) the BSDE with parameters $(\bar{f},\xi)$ has a unique solution. This implies that \autoref{main} holds. Let us emphasize that, applying the above-mentioned proof strategy, we can get a uniqueness result for the BSDEs with the generators that have linear growth in $z$ (see \autoref{main1}).
\ms

The paper is organized as follows. In Section 2, we introduce some notations and definitions that will be used in the sequel. The uniqueness of the solution to quadratic BSDE with the terminal condition that has bounded Malliavin derivative is obtained in Section 3. In Section 4, we prove the uniqueness of the solution for BSDE when the generator is only continuous with linear growth in $z$.

\section{Preliminaries}

Let us introduce some notations and spaces that will be used below.
For a positive integer $d$, any element $x \in \mathbb{R}^{d}$ will be identified to a column vector with $i$-th component $x^{i}$ and Euclidean norm denoted by $|x|$.
For $k\in\dbN$ and Euclidean spaces $\dbH$ and $\dbG$, denote by
$C^{k}_{b}(\dbH,\dbG)$ the set of functions of class $C^k$ from $\dbH$ to $\dbG$ whose partial
derivations of order less than or equal to $k$ are bounded.
For any real $p>1$, define
\begin{equation*}
\begin{aligned}
L_{\mathscr{F}_{T}}^{p}(\Omega ; \mathbb{H}) &=\left\{\xi: \Omega \rightarrow \mathbb{H} \mid \xi \text { is } \mathscr{F}_{T} \text {-measurable, }\|\xi\|_{L^{p}} \triangleq\left(\mathbb{E}|\xi|^{p}\right)^{\frac{1}{p}}<\infty\right\}, \\
L_{\mathscr{F}_{T}}^{\infty}(\Omega ; \mathbb{H}) &=\left\{\xi: \Omega \rightarrow \mathbb{H} \mid \xi \text { is } \mathscr{F}_{T} \text {-measurable, }\|\xi\|_{\infty} \triangleq \operatorname{esssup}_{\omega \in \Omega}|\xi(\omega)|<\infty\right\},
\end{aligned}
\end{equation*}
and for any $t\in[0,T)$, define
\begin{align*}
\ds L_\dbF^p(t,T;\dbH)\1n=&\ \1n\Big\{\f:[t,T]\1n\times\1n\Om\to\dbH\bigm|\f \hb{ is
$\dbF$-progressively measurable, }\\
\ns\ds&\qq
\|\f \|_{L_\dbF^p(t,T)}\deq\1n\(\dbE\int^T_t\1n|\f_s|^pds\)^{1\over p}\1n<\2n\i\Big\},\\
\ns\ds L_\dbF^\infty(t,T;\dbH)=&~
\Big\{\f:[t,T]\times\Om\to\dbH\bigm|\f \hb{ is $\dbF$-progressively measurable, }\\
\ns\ds&\qq
\|\f \|_{L_\dbF^\infty(t,T)}\deq\esssup_{(s,\o)\in[t,T]\times\Om}|\f_s(\o)|<\i\Big\},\\
\ns\ds S_\dbF^p(t,T;\dbH)=&~
\Big\{\f:[t,T]\times\Om\to\dbH\bigm|\f \hb{ is
$\dbF$-adapted, continuous, }\\
\ns\ds&\qq
\|\f \|_{S_\dbF^p(t,T)}\deq\Big\{\dbE\[\sup_{s\in[t,T]}|\f_s|^p\]\Big\}^{\frac{1}{p}}<\i\Big\},\\
\ns\ds S_\dbF^\infty(t,T;\dbH)=&~
\Big\{\f:[t,T]\times\Om\to\dbH\bigm|\f \hb{ is $\dbF$-progressively measurable, continuous }\\
\ns\ds&\qq
\|\f \|_{S_\dbF^\infty(t,T)}\deq\esssup_{(s,\o)\in[t,T]\times\Om}|\f_s(\o)|<\i\Big\},\\
\cH^2_{BMO}[t,T]=&~\Big\{\f \in L^2_\dbF(t,T;\dbH)\Bigm|\|\f\|_{\cH^2_{BMO}[t,T]} \triangleq \sup_{t \leqslant \t \leqslant T }\Big\|
\dbE_\t\[ \int_\t^T|\f_s|^2ds \] \Big\|_\i^{1\over 2}<\i\Big\}.
\end{align*}
where $\t$ is the stopping time and $\dbE_\t$ is the conditional expectation given $\sF_{\t}$.

\ms

Next, let us recall the notion of derivation on Wiener space,
%The following definition comes from Pardoux and Peng \cite{ref4},
 and for more information about Malliavin calculus we refer the readers to \cite{ref28}.
Denote by $\mathbb{S}$ the set of random variable $\xi$ of the form $\xi=\varphi(W(h^{1}),...,W(h^{n})),$
%
%\begin{equation*}
%  \xi=\varphi(W(h^{1}),...,W(h^{n})),
%\end{equation*}
%
where $\varphi \in C^{\infty}_{b}(\mathbb{R}^{n},\mathbb{R})$, $h^{i}\in L^{2}([0,T];\mathbb{R}^{d})$ and $W(h^{i})= \int_0^T h^{i}_t dW_t$ is the Wiener integral for $i=1,...,n$.
To such a random variable $\xi$, we associate a ``derivative process" $\{ D_t \xi;\ t\in[0,T] \}$ defined as
%\begin{equation*}
% D_t \xi\deq\sum_{i=1}^{n}\partial_{x_i} \varphi \big(W(h^{1}),...,W(h^{n})\big)h^{i}_t, \q~ t\in[0,T],
%\end{equation*}
$ D_t \xi\deq\sum_{i=1}^{n}\partial_{x_i} \varphi \big(W(h^{1}),...,W(h^{n})\big)h^{i}_t, \ \ t\in[0,T],$
whose components we denote by $D_{t}^{i} \xi$, for $i=1,...,d$.
For $\xi \in \mathbb{S}$, we define a kind of Sobolev norm by
%\begin{equation*}
%  \| \xi \|_{1,2}^{2}\deq\mathbb{E}\left[|\xi|^{2} + \int_0^T |D_{t}\xi|^{2} dt\right].
%\end{equation*}
$  \| \xi \|_{1,2}^{2}\deq\mathbb{E}\left[|\xi|^{2} + \int_0^T |D_{t}\xi|^{2} dt\right].$
It can be shown that the operator $D$ has a closed extension to the space $\mathbb{D}^{1,2}$, the closure of $\mathbb{S}$
with respect to the norm $\| \cdot \|_{1,2}^{2}$.
%Note that if $\xi \in \mathbb{D}^{1,2}$ is $\sF_{s}^{t}$-measurable, then we have $D_r \xi = 0$ for $r \in [0,T]\backslash (t,s]$.

\ms

Throughout this paper, to avoid the additional Malliavin regularization technicalities (see \cite{ref35} for more detail), we only consider the deterministic generator $f$ case. We work with the following
%As usual, we identify random variables that are equal $\mathbb{P}$-almost surely and accordingly, understand equalities and inequalities between them in the $\mathbb{P}$-almost sure sense.
\begin{definition}\label{bsdeddef}
A solution of the BSDE \eqref{bsde} is a pair $\left(Y_{t}, Z_{t}\right)_{0 \leqslant t \leqslant T}$ of adapted processes taking values in $\mathbb{R} \times \mathbb{R}^{d}$ such that $\int_{0}^{T}\left(\left|f\left(t, Y_{t}, Z_{t}\right)\right|+\left|Z_{t}\right|^{2}\right) d t<\infty$ and \eqref{bsde} holds for all $0 \leqslant t \leqslant T$, $\mathbb{P}$-a.s.
\end{definition}
This minimal definition will be completed later on by some further integrability assumptions. In order to simplify the notations, we sometimes write $(Y,Z)$ for the processes $\left(Y_{t}, Z_{t}\right)_{0 \leqslant t \leqslant T}$.

\section{The unique solution to the quadratic BSDE}

In this section, we study a unique solution for the quadratic BSDE \rf{bsde} with the continuous coefficient and the terminal condition that has bounded Malliavin derivative.

\ms

%\subsection{The bounded case}
%\tb{Now, we give some propositions in the literature of quadratic BSDEs.}
We work with the following assumptions:
\begin{itemize}
  \item [(A1)] The bounded terminal condition $\xi$ is in $\mathbb{D}^{1,2}$ and there exists a positive constant $L $ such that $\|\xi\|_{\infty} \leqslant L$ and $\left|D_{t}^{i} \xi\right| \leqslant L$, $d \mathbb{P} \otimes d t  $-a.e., for all $i=1, \ldots, d $.
  \item [(A2)] The generator $f$ is Borel measurable and continuous with respect to $y,z$, and there exists a positive constant $L$ such that for all $t\in[0,T]$ $y,y'\in\dbR$ and $z\in\dbR^d$
   \begin{align*}
  \ds |f(t,y,z)|  \les L(1+|y|+|z|^2) \text{ and } |f(t,y,z)-f(t,y',z)|  \les L|y-y'|.
   \end{align*}

  \item [(A3)] There is a positive constant $K$ such that for all $t\in[0,T]$, $y\in\dbR$ and $z,z'\in\dbR^d$,
      $$|f(t,y,z)-f(t,y,z')| \les K(1+|z|+|z'|)|z-z'|.$$
\end{itemize}

\ms

The following proposition was first obtained by Theorem 2.3 of \cite{ref25}, and some extended results can refer to \cite{ref6}, etc.

\begin{proposition}\label{BSDEmax}
If Assumptions (A1)-(A2) hold, then the BSDE \eqref{bsde} has the maximal solution $\left(\overline{Y}, \overline{Z}\right)$ (resp. the minimal solution $\left(\underline{Y}, \underline{Z}\right)$) in $S_{\mathbb{F}}^{\infty}(0, T ; \mathbb{R}) \times L_{\mathbb{F}}^{2}\left(0, T ; \mathbb{R}^{d}\right)$, i.e. for any solution $(Y,Z)$ of the BSDE \eqref{bsde}, $\overline{Y} \geqslant Y $ (resp. $\underline{Y} \leqslant Y $), $\mathbb{P}$-a.s..
\end{proposition}

\begin{remark}
Notice that if $\left(\overline{Y}, \overline{Z}\right)$ and $(\hat{Y},\hat{Z})$ are both maximal solutions of BSDE \eqref{bsde}, then necessarily $\overline{Y} \equiv \hat{Y}$ and $\overline{Z} \equiv \hat{Z}$.
\end{remark}

\ms

Let us recall the comparison principle which is one of the most important properties of BSDEs. The following proposition which is obtained under quite general conditions is an extension of Theorem 7.3.1 of \cite{ref26}.
\begin{proposition}\label{BSDEcomparison}
Assume that $\xi$ and $f$  satisfy Assumptions (A1)-(A3), let $(Y,Z) \in  S_{\mathbb{F}}^{\infty}(0, T ; \mathbb{R}) \times L_{\mathbb{F}}^{2}\left(0, T ; \mathbb{R}^{d}\right)$ be the unique solution of the BSDE \eqref{bsde}, and $(Y',Z') \in  S_{\mathbb{F}}^{\infty}(0, T ; \mathbb{R}) \times L_{\mathbb{F}}^{2}\left(0, T ; \mathbb{R}^{d}\right)$ be a solution of the BSDE with parameters $(\xi',f')$ satisfying Assumptions (A1)-(A2). If $\xi \geqslant \xi'$, $\mathbb{P}$-a.s., and $f(t,Y'_t,Z'_t) \geqslant f'(t,Y'_t,Z'_t)$, $d\mathbb{P} \otimes dt$-a.e., then $Y_t \geqslant Y'_t$, $0 \leqslant t\leqslant T$, $\mathbb{P}$-a.s.
%$\overline{Y} \geqslant Y $ (resp. $\underline{Y} \leqslant Y $), $\mathbb{P}$-a.s..
\end{proposition}

\begin{proof}
First, under Assumptions (A1)-(A3), since the solution $(Y,Z)$ of the BSDE \eqref{bsde} in $S_{\mathbb{F}}^{\infty}(0, T ; \mathbb{R}) \times L_{\mathbb{F}}^{2}\left(0, T ; \mathbb{R}^{d}\right)$, by Proposition 2.1 of \cite{ref23}, we have $Z \in \mathcal{H}_{B M O}^{2}[0, T]$. If the parameters $(\xi',f')$ satisfy Assumptions (A1)-(A2) and $(Y',Z') \in  S_{\mathbb{F}}^{\infty}(0, T ; \mathbb{R}) \times L_{\mathbb{F}}^{2}\left(0, T ; \mathbb{R}^{d}\right)$, by Proposition 1.1 of \cite{ref67}, we can have $Z' \in \mathcal{H}_{B M O}^{2}[0, T]$.
%Barrieu, Cazanave and El Karoui \cite{ref67}
\ms
Next, to simply notation we restrict ourselves to case $d = 1$. Denote
$$
\begin{gathered}
\Delta Y \triangleq Y-Y', \ \ \Delta Z \triangleq  Z-Z', \ \  \Delta \xi \triangleq \xi-\xi', \ \
\Delta f \triangleq f-f'.
\end{gathered}
$$
Then, by the classical linearization argument, we have
\begin{equation}\label{LBSDE}
\begin{aligned}
\Delta Y_{t} = \Delta \xi + \int_{t}^{T}\left[a_{r} \Delta Y_{r}+b_{r} \Delta Z_{r}+\Delta f\left(r, Y'_{r}, Z'_{r}\right)\right] d r -\int_{t}^{T} \Delta Z_{r} d W_{r}, \quad  0 \leqslant t \leqslant T,
\end{aligned}
\end{equation}
where the processes $a_{r}$ and $b_{r}$ are defined by
$$
a_{r} \triangleq\left\{\begin{array} { c l }
{ \frac { f ( r , Y _ { r } , Z _ { r } ) - f ( r , Y _ { r } ^ { \prime } , Z _ { r } ) } { \Delta Y _ { r } } , } & { \text { if  } \Delta Y _ { r } \neq 0, } \\
{ 0 , } & { \text { if  } \Delta Y _ { r } = 0. }
\end{array} \quad \text {and} \quad
b _ { r } \triangleq \left\{\begin{array}{cl}
\frac{f\left(r, Y_{r}^{\prime}, Z_{r}\right)-f\left(r, Y_{r}^{\prime}, Z_{r}^{\prime}\right)}{\Delta Z_{r}}, & \text { if  } \Delta Z_{r} \neq 0, \\
0, & \text { if  } \Delta Z_{r}=0.
\end{array}\right.\right.
$$
Now since the generator $f$ satisfies Assumptions (A2)-(A3), it implies that
$$|a_r| \leqslant L \quad \text{  and  } \quad |b_r| \leqslant K(1+|Z_r|+|Z'_r|).$$
Solving the linear BSDE \eqref{LBSDE} leads to
$$
\Gamma_{t} \Delta Y_{t}=\mathbb{E}\left[\Delta \xi \Gamma_{T}+\int_{t}^{T} \Gamma_{r} \Delta  f\left(r, Y'_{r}, Z'_{r}\right) d r \mid \sF_{t}\right],
$$
where $\Gamma$ is the adjoint (positive) process of the linear BSDE \eqref{LBSDE}, defined by the following SDE
$$
d \Gamma_{t} =\Gamma_{t}\left( a_t dt + b_t dW_t\right), \quad \Gamma_{0}=1.
$$
For $a$ is a bounded process and $b \in \mathcal{H}_{B M O}^{2}[0, T]$, then the $\Gamma$ is well-defined (see \cite{ref41}).

Besides, $\Delta \xi$ and $\Delta  f\left(r, Y'_{r}, Z'_{r}\right)$ are nonnegative, then $\Delta Y$ is nonnegative. Thus \autoref{BSDEcomparison} is proved.
\end{proof}

\begin{remark}
%El Karoui, Peng and Quenez \cite{ref8}
The above proof method comes from Theorem 2.2 of \cite{ref8} and Theorem 7.3.1 of \cite{ref26}. The terminal condition $\xi$ that has bounded Malliavin derivative in Assumption (A1) can be dropped. We can also get the strict comparison theorem similar to \cite{ref8}.

\end{remark}

\ms

The following result gives the regular properties of the solutions of the BSDEs, which comes from \cite{ref35}. For the convenience of readers, we give a brief proof.

\begin{proposition}\label{BSDEbounded}
If Assumptions (A1)-(A3) hold, then the BSDE \eqref{bsde} has a unique solution $\left(Y, Z\right)$ in $S_{\mathbb{F}}^{\infty}(0, T ; \mathbb{R}) \times L_{\mathbb{F}}^{\infty}\left(0, T ; \mathbb{R}^{d}\right)$, and $\left|Y_{t}\right| \leqslant(L+1) e^{L(T-t)}-1$, $\left|Z_{t}^{i}\right| \leqslant L e^{L(T-t)}$, $d\mathbb{P} \otimes dt$-a.e. for all $i=1, \ldots, d.$
\end{proposition}

\begin{proof}
Under Assumptions (A1)-(A3), we set $\rho (x) \triangleq  K\left(1+2x\right)$ which is a nondecreasing function. Since $\rho\left(|z| \vee\left|z^{\prime}\right|\right) = K\left[1+2\left(|z| \vee\left|z^{\prime}\right|\right)\right] \geqslant K(1+|z|+|z'|)$, we have
$$
\begin{aligned}
\left|f(t, y, z)-f\left(t, y, z^{\prime}\right)\right| & \leqslant K\left(1+|z|+\left|z^{\prime}\right|\right)\left|z-z^{\prime}\right| \\
& \leqslant  \rho\left(|z| \vee\left|z^{\prime}\right|\right)\left|z-z^{\prime}\right|.
\end{aligned}
$$
Then the terminal condition $\xi$ and the deterministic generator $f$ satisfy the assumptions of the corollary 2.8 of \cite{ref35}, so we can get that there exists a unique solution $\left(Y, Z\right)$ in $S_{\mathbb{F}}^{\infty}(0, T ; \mathbb{R}) \times L_{\mathbb{F}}^{\infty}\left(0, T ; \mathbb{R}^{d}\right)$ to the BSDE \eqref{bsde} and $\left|Y_{t}\right| \leqslant(L+1) e^{L(T-t)}-1$, $\left|Z_{t}^{i}\right| \leqslant L e^{L(T-t)}$, $d\mathbb{P} \otimes dt$-a.e. for all $i=1, \ldots, d.$

\end{proof}

\begin{remark}
It should be pointed out that the bounds of the solution $\left(Y, Z\right)$ are not depend on the constant $K$ of (A3). This is important for us to obtain the uniformly bounded solutions of the BSDEs with the approximation coefficients in Proposition \ref{BSDEbijin}.
\end{remark}

\ms

Now we present the regularization of the generator of BSDE \eqref{bsde} through sup-convolution in the following lemma, which is useful for us to study the quadratic BSDEs later.
\begin{lemma}\label{BSDEfbijin}
Let the function $f$ satisfies Assumption (A2), then the sequence of functions
\begin{equation}
\label{f-n}
f_{n}(t, y, z) \triangleq \sup_{v \in \mathbb{R}^{d}}\left\{f(t, y, v) - n|z-v|^{2}\right\},
\end{equation}
is well defined for $n \geqslant 2 L .$ Moreover, it satisfies
\begin{itemize}
  \item [(1)]  For any $t \in[0, T]$, $y \in \mathbb{R}$ and $z \in \mathbb{R}^{d}, \left|f_{n}(t, y, z)\right| \leqslant L\left(1+|y|+2|z|^{2}\right)$.
  \item [(2)] For any $t \in[0, T]$, $y \in \mathbb{R}$ and $z \in \mathbb{R}^{d}, f_{n}(t, y, z)$ is decreasing in $n$.
  \item [(3)] If $z_{n} \rightarrow z$ as $n \rightarrow \infty$, then $f_{n} (t, y, z_{n} ) \rightarrow f(t, y, z)$.
  \item [(4)] For any $t \in[0, T]$, $y_{1}, y_{2} \in \mathbb{R}$ and $z_{1}, z_{2} \in \mathbb{R}^{d}$ and for sufficiently large $n$,
$$
\left|f_{n} (t, y_{1}, z_{1} )-f_{n} (t, y_{2}, z_{2} )\right| \leqslant L\left|y_{1}-y_{2}\right|+n\left(1+\left|z_{1}\right|+\left|z_{2}\right|\right)\left|z_{1}-z_{2}\right|.
$$
\end{itemize}
\end{lemma}

\begin{proof}
We consider the case of the generator $f$ independent of $t$ and $y$, i.e., $f(\cdot)=f(z)$, and the case of $f(t, y, z)$ could be proved similar without substantial difference. Let $h: \mathbb{R}^{d} \rightarrow \mathbb{R}$ be a continuous function and $|h(z)| \leqslant L\left(1+|z|^{2}\right)$, then for any $n \geqslant 2 L$, we define
$$
h_{n}(z) \triangleq \sup_{v \in \mathbb{R}^{d}}\left\{h(v) - n|z-v|^{2}\right\}.
$$
From the definition, we see that $h_{n}(z) \geqslant h(z) \geqslant -L\left(1+|z|^{2}\right)$ and
$$
\begin{aligned}
h_{n}(z) & \leqslant \sup _{v \in \mathbb{R}^{d}}\left\{ L+L|v|^{2}- n|z-v|^{2}\right\} \\
&=\sup _{v \in \mathbb{R}^{d}}\left\{ L+L|v|^{2}-2 L|z-v|^{2}-(n-2 L)|z-v|^{2}\right\} \\
& \leqslant \sup _{v \in \mathbb{R}^{d}}\left\{ L+2 L|z|^{2}-(n-2 L)|z-v|^{2}\right\} \\
&= L+2 L|z|^{2}.
\end{aligned}
$$
So $\left|h_{n}(z)\right| \leqslant L\left(1+2|z|^{2}\right)$, from which item (1) holds. Moreover, item (2) comes from the definition of $h_{n}(\cdot)$ directly.

\ms

Now we prove the item (3), and consider $z_n \to z$, then for every $n$, there exists $v_{n} \in \mathbb{R}^{d}$ such that
$$
\begin{aligned}
h\left(z_{n}\right) \leqslant h_{n}\left(z_{n}\right) & \leqslant h\left(v_{n}\right)-n\left|z_{n}-v_{n}\right|^{2}+\frac{1}{n} \\
& \leqslant L+L\left|v_{n}\right|^{2}-n\left|z_{n}-v_{n}\right|^{2}+\frac{1}{n} \\
& \leqslant L+2 L\left|z_{n}\right|^{2}-(n-2 L)\left|z_{n}-v_{n}\right|^{2}+\frac{1}{n}.
\end{aligned}
$$
Since $h\left(z_{n}\right)$ is bounded, we deduce that $\lim _{n \rightarrow \infty} \sup (n-2 L)\left|z_{n}-v_{n}\right|^{2}<\infty$. In particular, when $v_{n} \rightarrow z$, we have
$$
\lim _{n \rightarrow \infty} \sup (n-2 L)\left|z_{n}-z\right|^{2}<\infty.
$$
Moreover, we have that
$$
h\left(z_{n}\right) \leqslant h_{n}\left(z_{n}\right) \leqslant h\left(v_{n}\right)+\frac{1}{n}
$$
from which $h_{n}\left(z_{n}\right) \rightarrow h(z)$ when $z_{n} \rightarrow z$, this implies that item (3) holds.

\ms

In order to prove item (4), for any $z \in \IR^d$, we take $\epsilon>0$ and consider $v_{\epsilon} \in \mathbb{R}^{d}$ (in fact for sufficiently large $n$, we can assumed that $\left.\left|z-v_{\epsilon}\right| \leqslant \frac{1}{2}\right)$ such that
$$
\begin{aligned}
h_{n}(z) & \leqslant h\left(v_{\epsilon}\right)-n\left|z-v_{\epsilon}\right|^{2}+\epsilon \\
&= h\left(v_{\epsilon}\right)-n\left|v-v_{\epsilon}\right|^{2}+n\left|v-v_{\epsilon}\right|^{2}-n\left|z-v_{\epsilon}\right|^{2}+\epsilon \\
& \leqslant h_{n}(v)+n\left|v-v_{\epsilon}\right|^{2}-n\left|z-v_{\epsilon}\right|^{2}+\epsilon \\
& \leqslant h_{n}(v)+n|z-v|\left(\left|v-v_{\epsilon}\right| + \left|z-v_{\epsilon}\right|\right)+\epsilon \\
& \leqslant h_{n}(v)+n|z-v|(1+|z|+|v|)+\epsilon.
\end{aligned}
$$
Therefore, interchanging the role of $z$ and $v$, and since $\epsilon$ is arbitrary we can get
$$
\left|h_{n}(z)-h_{n}(v)\right| \leqslant n(1+|z|+|v|)|z-v|.
$$
Finally, for any $y_{1}, y_{2} \in \mathbb{R}, z \in \mathbb{R}^{d}$ and $t \in[0, T]$, by the inequality
$$
\left|\sup _{v \in \mathbb{R}^{d}} f(v)-\sup _{v \in \mathbb{R}^{d}} g(v)\right| \leqslant \sup _{v \in \mathbb{R}^{d}}|f(v)-g(v)|
$$
and the Lipschitz condition of $f$ with respect to the spatial variable $y$, we get that
$$
\left|f_{n}\left(t, y_{1}, z\right)-f_{n}\left(t, y_{2}, z\right)\right| \leqslant L\left|y_{1}-y_{2}\right|
$$
from which the desired result follows.
\end{proof}

\begin{remark}\label{smooth}
The idea of the proof of the above lemma is inspired by Lemma 1 in \cite{ref30}.
\end{remark}

\ms

In the following, for positive number $n \geqslant 2L$, we consider the following BSDE:

\begin{equation}\label{BSDEjieduan}
Y_t^n = \xi + \int^T_t f_n(r,Y_r^n,Z_r^n)dr - \int_t^T Z_r^n dW_r, \q~    0 \les t \les T,
\end{equation}
where $f_n(t,y,z)$ is defined by \eqref{f-n}. For the above BSDE \rf{BSDEjieduan}, we have the following properties concerning the solutions.

\begin{proposition}\label{BSDEbijin}
Under Assumptions (A1)-(A2), then for any sufficiently large $n \in \dbN$,
the BSDE \eqref{BSDEjieduan} admits a unique solution $\left(Y^{n}, Z^{n}\right) \in S_{\mathbb{F}}^{\infty}(0, T ; \mathbb{R}) \times L_{\mathbb{F}}^{\infty}\left(0, T ; \mathbb{R}^{d}\right)$, and $Y^{n}$ is decreasing in $n$.
 Moreover, $\left(Y^{n}, Z^{n}\right) \rightarrow\left(\overline{Y}, \overline{Z}\right)$ in the space $S_{\mathbb{F}}^{2}(0, T ; \mathbb{R}) \times L_{\mathbb{F}}^{2}\left(0, T ; \mathbb{R}^{d}\right)$ as $n \rightarrow \infty$, where $(\overline{Y}, \overline{Z})$ in $S_{\mathbb{F}}^{\infty}(0, T ; \mathbb{R}) \times L_{\mathbb{F}}^{\infty}\left(0, T ; \mathbb{R}^{d}\right)$ is the maximal solution of BSDE \eqref{bsde}.
\end{proposition}

\begin{proof}
Under Assumptions (A1)-(A2), by \autoref{BSDEfbijin}, we know that for any sufficiently large  $n \in \dbN $, the function $f_n(r,y,z)$ and terminal condition $\xi$ satisfy the conditions of \autoref{BSDEbounded}, then the BSDE \eqref{BSDEjieduan} has a unique solution $(Y^{n},Z^{n})$ and
$$
\begin{gathered}
\left|Y_{t}^n \right| \leqslant(L+1) e^{L(T-t)}-1, \text { for all } t \in[0, T], \ \
\left|Z_{t}^{n,i}\right| \leqslant L e^{L(T-t)}, \ \  d\mathbb{P} \otimes dt   \text {-a.e. for all } i=1, \ldots, d.
\end{gathered}
$$
Note that the sequence $\{ f_n(r,y,z) \}_{n}$ is decreasing, so from the comparison principle (see Theorem 7.3.1 of \cite{ref26}), we have $Y^{n} \geqslant Y^{n+1}$, $\mathbb{P}$-a.s..

\ms

Let us recapitulate what we have obtained. The sequence $\{ f_n(t,y,z) \}_{n}$ converges locally uniformly in $(y,z)$ to the generator $f(t,y,z)$, and $\left|f_{n}(t, y, z)\right| \leqslant L\left(1+|y|+2|z|^{2}\right)$. The BSDE \eqref{BSDEjieduan} with parameters $(f_{n},\xi) $ has a unique bounded solution $(Y^{n},Z^{n})$ such that the sequence $\{ Y^n \}_{n}$ is decreasing. Therefore, applying the monotone stability proposition (see Proposition 2.4 of \cite{ref25} or Lemma 3 of \cite{ref22}), we have $(Y^{n},Z^{n})\rightarrow (\overline{Y}, \overline{Z})$ in the space $S_\dbF^2(0,T;\dbR) \times L_\dbF^2(0,T;\dbR^d)$ as $n\rightarrow \infty$, and $(\overline{Y}, \overline{Z})$ is a solution of the BSDE \eqref{bsde} with parameters $(f,\xi) $. The solution $(\overline{Y}, \overline{Z})$ is the maximal solution of the BSDE \eqref{bsde} was followed by $f_{n}(t, y, z)  \geqslant f(t, y, z) $ and the comparison principle (see \autoref{BSDEcomparison}). Finally, since the bounds of the solution $(Y^{n},Z^{n})$ does not depend on $n$, we have $(\overline{Y}, \overline{Z}) \in  S_\dbF^\i(0,T;\dbR) \times L_\dbF^\infty(0,T;\dbR^d)$, and $\left|\overline{Y}_{t}\right| \leqslant(L+1) e^{L(T-t)}-1$, $\left|\overline{Z}_{t}^{i}\right| \leqslant L e^{L(T-t)}$, $d\mathbb{P} \otimes dt$-a.e. for all $i=1, \ldots, d.$
\end{proof}

\ms

Furthermore, to study the minimal solution of BSDE \eqref{bsde}, we apply the similar above arguments. Firstly, we give the following lemma involving the regularization of the generator $f$ by inf-convolution techniques. The proof of \autoref{BSDEfbijin1} is similar to \autoref{BSDEfbijin}, so we omit the proof.

\begin{lemma}\label{BSDEfbijin1}
Let the function $f$ satisfies Assumption (A2), then the sequence of functions
\begin{equation}
\label{f-n1}
\hat{f}_{n}(t, y, z) \triangleq \inf _{v \in \mathbb{R}^{d}}\left\{f(t, y, v)+n|z-v|^{2}\right\},
\end{equation}
is well defined for $n \geqslant 2 L .$ Moreover, it satisfies
\begin{itemize}
  \item [(1)]  For any $t \in[0, T]$, $y \in \mathbb{R}$ and $z \in \mathbb{R}^{d}, \left|\hat{f}_{n}(t, y, z)\right| \leqslant L\left(1+|y|+2|z|^{2}\right)$.
  \item [(2)] For any $t \in[0, T]$, $y \in \mathbb{R}$ and $z \in \mathbb{R}^{d}, \hat{f}_{n}(t, y, z)$ is increasing in $n$.
  \item [(3)] If $z_{n} \rightarrow z$ as $n \rightarrow \infty$, then $\hat{f}_{n} (t, y, z_{n} ) \rightarrow f(t, y, z)$.
  \item [(4)] For any $t \in[0, T]$, $y_{1}, y_{2} \in \mathbb{R}$ and $z_{1}, z_{2} \in \mathbb{R}^{d}$ and for sufficiently large $n$,
$$
\left|\hat{f}_{n} (t, y_{1}, z_{1} )-\hat{f}_{n} (t, y_{2}, z_{2} )\right| \leqslant L\left|y_{1}-y_{2}\right|+n\left(1+\left|z_{1}\right|+\left|z_{2}\right|\right)\left|z_{1}-z_{2}\right|.
$$
\end{itemize}
\end{lemma}

\ms

In the following, for positive number $n \geqslant 2L$, we consider the following BSDE:

\begin{equation}\label{BSDEjieduan1}
Y_t^n = \xi + \int^T_t \hat{f}_n(r,Y_r^n,Z_r^n)dr - \int_t^T Z_r^n dW_r, \q~    0 \les t \les T,
\end{equation}
where $\hat{f}_n(t,y,z)$ is defined by \eqref{f-n1}. For the above BSDE \rf{BSDEjieduan1},  by the similar arguments in \autoref{BSDEbijin}, we have the following proposition.

\begin{proposition}\label{BSDEbijin1}
Under Assumptions (A1)-(A2), then for any sufficiently large $n \in \dbN$,
the BSDE \eqref{BSDEjieduan1} admits a unique solution $\left(Y^{n}, Z^{n}\right) \in S_{\mathbb{F}}^{\infty}(0, T ; \mathbb{R}) \times L_{\mathbb{F}}^{\infty}\left(0, T ; \mathbb{R}^{d}\right)$, and $Y^{n}$ is increasing in $n$.
 Moreover, $\left(Y^{n}, Z^{n}\right) \rightarrow \left(\underline{Y}, \underline{Z}\right)$ in the space $S_{\mathbb{F}}^{2}(0, T ; \mathbb{R}) \times L_{\mathbb{F}}^{2}\left(0, T ; \mathbb{R}^{d}\right)$ as $n \rightarrow \infty$, where $(\underline{Y}, \underline{Z})$ in $S_{\mathbb{F}}^{\infty}(0, T ; \mathbb{R}) \times L_{\mathbb{F}}^{\infty}\left(0, T ; \mathbb{R}^{d}\right)$ is the minimal solution of BSDE \eqref{bsde}.
\end{proposition}

\ms

Now, we give the main result in this section.

\begin{theorem}
\label{main}
Under Assumptions (A1)-(A2), the BSDE \eqref{bsde} has a unique adapted solution $(Y,Z)$ in $S_{\mathbb{F}}^{\infty}(0, T ; \mathbb{R}) \times L_{\mathbb{F}}^{\infty}\left(0, T ; \mathbb{R}^{d}\right)$.
\end{theorem}

\begin{proof}

First of all, by \autoref{BSDEbijin} and \autoref{BSDEbijin1}, the BSDE \eqref{bsde} has the maximal and minimal solution $\left(\overline{Y}, \overline{Z}\right)$ and $\left(\underline{Y}, \underline{Z}\right)$ in $S_{\mathbb{F}}^{\infty}(0, T ; \mathbb{R}) \times L_{\mathbb{F}}^{\infty}\left(0, T ; \mathbb{R}^{d}\right)$, respectively. The bound of $\overline{Z}$ and $\underline{Z} $ is denoted by the constant $M$.

\ms

Next, we define
$$
\bar{f}(t, y, z) \triangleq  \begin{cases}f(t, y, z), & \text { if }|z| \leqslant M; \\ f(t, y, \frac{M z}{|z|}), & \text { if }|z| > M.\end{cases}
$$
Then under Assumption (A2), the generator $\bar{f}(t,y,z)$ is Lipschitz continuous in $y$ and uniformly continuous in $z$. Furthermore, by Theorem 2 of \cite{ref55}, the BSDE with parameters $(\bar{f},\xi)$ has a unique solution. Now we can prove that the maximal solution $\left(\overline{Y}, \overline{Z}\right)$ of the BSDE \eqref{bsde} must be equal to the minimal solution $\left(\underline{Y}, \underline{Z}\right)$, since solving the BSDE \eqref{bsde} with the generator $\bar{f}$ that coincides with the generator $f$ for $|z| \leqslant M$, where $M$ is the bound on $\overline{Z}$ and $\underline{Z}$.
%$\left\{z|| z^{i} \mid \leqslant M\right.$, for all $\left.i=1, \ldots, d .\right\}$

\end{proof}

\begin{remark}
In the above theorem, we see that if we assume that the terminal condition $\xi$ has bounded Malliavin derivative, even though relaxing some assumptions on the generator $f$, we can also get the uniqueness result for the BSDE.
\end{remark}

\begin{remark}
For the unbounded terminal condition case, similar results will be in our coming paper. In fact, we only need the comparison theorem (see \autoref{BSDEcomparison}) holds for the solution $(Y',Z')$ in the appropriate space.
\end{remark}

\section{The uniqueness solution to the BSDEs with linear growth}

In this section we state and prove a uniqueness result for the BSDE \eqref{bsde}, whose generator $f(t,y,z)$ is Lipschitz in $y$ and only continuous with linear growth in $z$. The following theorem is the main result of this section.

\begin{theorem}
\label{main1}
Assume that the following conditions hold:
\begin{itemize}
    \item [(B1)] The terminal condition $\xi$ is in $\mathbb{D}^{1,2}$ and there exists a positive constant $L $ such that $\left|D_{t}^{i} \xi\right| \leqslant L$, $d \mathbb{P} \otimes d t  $-a.e., for all $i=1, \ldots, d $;
    \item [(B2)] The generator $f$ is Borel measurable and continuous with respect to $y,z$, and there exists a positive constant $L$ such that for all $t\in[0,T]$ $y,y'\in\dbR$ and $z\in\dbR^d$
   \begin{align*}
  \ds |f(t,y,z)|  \les L(1+|y|+|z|) \text{ and } |f(t,y,z)-f(t,y',z)|  \les L|y-y'|.
   \end{align*}

\end{itemize}
Then the BSDE \eqref{bsde} has a unique adapted solution $(Y,Z)$ in $S_{\mathbb{F}}^{2}(0, T ; \mathbb{R}) \times L_{\mathbb{F}}^{\infty}\left(0, T ; \mathbb{R}^{d}\right)$.
\end{theorem}

\begin{proof}
Since the only novelty is the gathering of known results and applying the proof strategy in Section 3, we will only sketch the proof.

\ms

1. First of all, since the terminal condition $\xi$ satisfies Assumption (B1), by Lemma 2.5 of \cite{ref35}, we have $\xi \in L_{\mathscr{F}_{T}}^{2}(\Omega ; \mathbb{R})$. Moreover, since the generator $f$ satisfies Assumption (B2), we regularize $f$ by
$$\check{f}_{n}(t, y, z) \triangleq \inf _{v \in \mathbb{R}^{d}}\left\{f(t, y, v)+n|z-v|\right\},$$
and for $n \geqslant L$, the sequence of functions $\check{f}_{n}(t, y, z)$ satisfies the following properties (see Lemma 1 in \cite{ref30})
\begin{itemize}
  \item [{\rm (i)}] For any $t\in[0,T]$, $y\in\IR$ and $z\in\IR^d$,
$ |\check{f}_n(t,y,z)| \les L(1+|y|+|z|)$;
  \item [{\rm (ii)}] For any $t\in[0,T]$, $y\in\IR$ and $z\in\IR^d$,
$\check{f}_n(t,y,z)$ is increasing in $n$;
  \item [{\rm (iii)}] If $z_n \to z$ as $n\to\infty$, then
$\check{f}_n(t,y,z_n)\to f(t,y,z)$;
  \item [{\rm (iv)}] For any $t\in[0,T]$, $y_1,y_2\in\IR$ and $z_1,z_2\in\IR^d$, $|\check{f}_n(t, y_1,z_1)-\check{f}_n(t, y_2,z_2)|\les L|y_1-y_2|+ n|z_1-z_2|.$
\end{itemize}

\ms

2. For any $n \geqslant L$, we consider the following BSDE
\begin{equation}\label{BSDEjieduan2}
\check{Y}_t^n = \xi + \int^T_t \check{f}_n(r,\check{Y}_r^n,\check{Z}_r^n)dr - \int_t^T \check{Z}_r^n dW_r, \q~    0 \les t \les T,
\end{equation}
%El Karoui, Hamad¨¨ne and Matoussi \cite{ref6}
with the generator $\check{f}_n$. Now $(\xi , \check{f}_n)$ satisfy the assumptions of Theorem 1 of \cite{ref30} (or Theorem 8.4 of \cite{ref6}), then the unique solution $(\check{Y}^n,\check{Z}^n)$ of the BSDE \eqref{BSDEjieduan2} converges to the minimal solution $(\check{Y},\check{Z})$ of the BSDE \eqref{bsde}, in the space $L_{\mathbb{F}}^{2}(0, T ; \mathbb{R}) \times L_{\mathbb{F}}^{2}\left(0, T ; \mathbb{R}^{d}\right)$ as $n \rightarrow \infty$. In addition, by Theorem 2.2 and Remark 2.3 of \cite{ref35}, we have that the solution $(\check{Y}^n,\check{Z}^n) \in S_{\mathbb{F}}^{2}(0, T ; \mathbb{R}) \times L_{\mathbb{F}}^{\infty}\left(0, T ; \mathbb{R}^{d}\right)$, and $\left|\check{Z}_{t}^{n,i}\right| \leqslant L e^{L(T-t)}$, $d\mathbb{P} \otimes dt$-a.e. for all $i=1, \ldots, d$. So we have $\left|\check{Z}_{t}^{i}\right| \leqslant L e^{L(T-t)}$, $d\mathbb{P} \otimes dt$-a.e. for all $i=1, \ldots, d$.

If we regularize $f$ by $\breve{f}_{n}(t, y, z) \triangleq \sup _{v \in \mathbb{R}^{d}}\left\{f(t, y, v)-n|z-v|\right\}$, and applying the similar arguments, then under Assumptions (B1)-(B2) the BSDE \eqref{bsde} admits the maximal solution $(\breve{Y},\breve{Z}) \in S_{\mathbb{F}}^{2}(0, T ; \mathbb{R}) \times L_{\mathbb{F}}^{\infty}\left(0, T ; \mathbb{R}^{d}\right)$, and $\left|\breve{Z}_{t}^{i}\right| \leqslant L e^{L(T-t)}$, $d\mathbb{P} \otimes dt$-a.e. for all $i=1, \ldots, d$.

\ms

3. Now let us summarize the conclusions we have obtained. The BSDE \eqref{bsde} admits the maximal $(\breve{Y},\breve{Z})$ and minimal solution $(\check{Y},\check{Z})$ in $S_{\mathbb{F}}^{2}(0, T ; \mathbb{R}) \times L_{\mathbb{F}}^{\infty}\left(0, T ; \mathbb{R}^{d}\right)$, respectively.
Moreover, for all $i=1, \ldots, d$, $\left|\breve{Z}^i\right|$ and $\left|\check{Z}^{i}\right|$ are bounded by constant $C \triangleq L e^L$. Since $f$ satisfies Assumption (B2), then $f$ is Lipschitz continuous in $y$ and uniformly continuous in $z$ for $z \in [-C,C]^d $. Furthermore, by Theorem 2 of \cite{ref55}, we have $(\breve{Y},\breve{Z}) \equiv (\check{Y},\check{Z})$.

\end{proof}

\section*{Acknowledgments}

This work is supported by National Key R\&D Program of China (Grant No. 2018YFA0703900), the National Natural Science Foundation of China (Grant Nos. 11871309 and 11371226), the Shandong Provincial Natural Science Foundation (No. ZR2019ZD41). Zhi Yang is supported by the State Scholarship Fund from the China Scholarship Council (No. 201906220089). The main work have been done when this author was a visiting Ph.D. student at Carleton University. The hospitality of Carleton University is appreciated.

%{\bf Acknowledgment:} The first author thanks their helpful discussions and suggestions, of course all errors belong to us.


\section*{References}
\begin{thebibliography}{99}
%\bibitem{ref67}P. Barrieu, N. Cazanave, N. El Karoui. Closedness results for BMO semi-martingales and application to quadratic BSDEs, C. R. Math. Acad. Sci. Paris, 2008, 346(15-16) 881--886.
\bibitem[Barrieu et al.(2008)]{ref67}Barrieu, P., Cazanave, N., El Karoui, N., 2008. Closedness results for BMO semi-martingales and application to quadratic BSDEs. C. R. Math. Acad. Sci. Paris 346(15), 881--886.

%\bibitem{ref40}P. Barrieu, N. El Karoui. Pricing, hedging, and designing derivatives with risk measures, Indifference Pricing: Theory and Applications edited by Ren¨¦ Carmona, Princeton University Press, 2009, 77--146.

%\bibitem{ref61}P. Barrieu, N. El Karoui. Monotone stability of quadratic semimartingales with applications to unbounded general quadratic BSDEs, Ann. Probab. (2013) 41(3) 1831-1863.

%\bibitem{ref7}J. Bismut. Conjugate convex functions in optimal stochastic control, J. Math. Anal. Appl. (1973) 44(2) 384--404.
%
%\bibitem{ref29}M. Bou\'{e}, P. Dupuis. A variational representation for certain functionals of Brownian motion, Ann. Probab. (1998) 26(4) 1641--1659.
%
%\bibitem{ref24}P. Boyle, S. Feng, W. Tian. Large Deviation Techniques and Financial Applications, Handbooks in Operations Research and Management Science, 2007 15 971--1000.
%
%\bibitem{ref14}P. Boyle, S. Feng, W. Tian, T. Wang. Robust stochastic discount factors, Rev. Financ. Stud. (2007) 21(3) 1077--1122.

%\bibitem{ref23}P. Briand, R. Elie. A simple constructive approach to quadratic BSDE with or without delay, Stoch. Process. Appl. (2013) 123(8) 2921--2939.
\bibitem[Briand and Elie(2013)]{ref23}Briand, P., Elie, R., 2013. A simple constructive approach to quadratic BSDE with or without delay. Stoch. Process. Appl. 123(8), 2921--2939.

%\bibitem{ref22}P. Briand, Y. Hu. BSDE with quadratic growth and unbounded terminal value, Probab. Theory Relat. Fields (2006) 136(4) 604--618.
\bibitem[Briand and Hu(2006)]{ref22}Briand, P., Hu, Y., 2006. BSDE with quadratic growth and unbounded terminal value. Probab. Theory Relat. Fields 136(4), 604--618.

%\bibitem{ref08}
%    P.~Briand, Y.~Hu. Quadratic BSDEs with convex generators and unbounded terminal conditions, Probab. Theory Relat. Fields (2008) 141(4) 543--567.
\bibitem[Briand and Hu(2008)]{ref08}Briand, P., Hu, Y., 2008. Quadratic BSDEs with convex generators and unbounded terminal conditions. Probab. Theory Relat. Fields 141(4), 543--567.

%\bibitem{ref51}P. Briand, J. Lepeltier, J. San Mart¨ªn. One-dimensional BSDE's whose coefficient is monotonic in y and non-lipschitz in z. Bernoulli (2007) 13(1) 80--91.
%J.-P. Lepeltier
%\bibitem{ref66}P. Briand, A. Richou. On the uniqueness of solutions to quadratic BSDEs with non-convex generators, Frontiers in stochastic analysis--BSDEs, SPDEs and their applications, Springer Proc. Math. Stat., 289, Springer, Cham, 2019 89--107.
\bibitem[Briand and Richou(2019)]{ref66}Briand, P., Richou, A., 2019. On the uniqueness of solutions to quadratic BSDEs with non-convex generators. Frontiers in stochastic analysis--BSDEs, SPDEs and their applications Springer Proc. Math. Stat. 289, 89--107.

%\bibitem{ref18}Z. Chen, J. Xiong. Large deviation principle for diffusion processes under a sublinear expectation, Sci. China Math. (2012) 55(11) 2205--2216.

%\bibitem{ref35}P. Cheridito, K. Nam. BSDEs with terminal conditions that have bounded Malliavin derivative, J. Funct. Anal. (2014) 266(3) 1257--1285.
\bibitem[Cheridito and Nam(2014)]{ref35}Cheridito, P., Nam, K., 2014. BSDEs with terminal conditions that have bounded Malliavin derivative. J. Funct. Anal. 266 (3), 1257--1285.

%\bibitem{ref9} M. Crandall, H. Ishii, P. Lions. User¡¯s guide to viscosity solutions of second order partial differential equations, Bull. Amer. Math. Soc. (1992) 27(1) 1--67.
%
%\bibitem{ref1}A. Cruzeiro, A. Gomes, L. Zhang. Asymptotic properties of coupled forward--backward stochastic differential equations, Stoch. Dynam. 14(3) (2014) 1450004.

%\bibitem{ref71}F. Delbaen, Y. Hu, X. Bao. Backward SDEs with superquadratic growth, Probab. Theory Relat. Fields (2011) 150 145--192.
%%Xiaobo Bao

%\bibitem{ref62}F. Delbaen, Y. Hu, A. Richou. On the uniqueness of solutions to quadratic BSDEs with convex generators and unbounded terminal conditions, Ann. Inst. Henri Poincare Probab. Stat. (2011) 47(2) 559--574.
\bibitem[Delbaen et al.(2011)]{ref62}Delbaen, F., Hu, Y., Richou, A., 2011. On the uniqueness of solutions to quadratic BSDEs with convex generators and unbounded terminal conditions. Ann. Inst. Henri Poincare Probab. Stat. 47(2), 559--574.

%\bibitem{ref20}A. Dembo, O. Zeitouni. Large Deviations Techniques and Applications, second ed., Springer--Verlag, New York, 1998.
%
%\bibitem{ref10}D. Duffie, L. Epstein. Stochastic differential utility, Econometrica: Vol.60 No.2 (1992) 353--394.

%\bibitem{ref68}N. El Karoui, S. Hamad¨¨ne. BSDEs and risk-sensitive control, zero-sum and nonzero-sum game problems of stochastic functional differential equations, Stoch. Process. Appl. (2003) 107 145¨C169.

%\bibitem{ref6}N. El Karoui, S. Hamadene, A. Matoussi. Backward stochastic differential equations and applications. Indifference Pricing: Theory and Applications edited by Ren¨¦ Carmona, Princeton University Press, 2009, 267--320.
\bibitem[El Karoui et al.(2009)]{ref6}El Karoui, N., Hamadene, S., Matoussi, A., 2009. Backward stochastic differential equations and applications. Indifference pricing: theory and applications, 267--320.
%\bibitem{ref8}N. El Karoui, S. Peng, M. Quenez. Backward stochastic differential equations in finance, Math. Finance (1997) 7(1) 1--71.
\bibitem[El Karoui et al.(1997)]{ref8}El Karoui, N., Peng, S.G., Quenez, M.C., 1997. Backward stochastic differential equations in finance. Math. Finance 7 (1), 1-71.

%\bibitem{ref38}L. Evans. Partial Differential Equations, second ed., AMS, Providence, 2010.

%\bibitem{ref16}E. Essaky. Large deviation principle for a backward stochastic differential equation with subdifferential operator, C. R. Acad. Sci. Paris, Ser. I 346 (2008) 75--78.

%\bibitem{ref69}S. Fan, Y. Hu, S. Tang. On the uniqueness of solutions to quadratic BSDEs with non-convex generators and unbounded terminal conditions, C. R. Math. Acad. Sci. Paris (2020) 358(2) 227--235.
\bibitem[Fan et al.(2020)]{ref69}Fan, S.J., Hu, Y., Tang, S.J., 2020.  On the uniqueness of solutions to quadratic BSDEs with non-convex generators and unbounded terminal conditions. C. R. Math. Acad. Sci. Paris 358(2), 227--235.

%\bibitem{ref55} S. Fan, L. Jiang. Existence and uniqueness result for a backward stochastic differential equation whose generator is Lipschitz continuous in y and uniformly continuous in z, J. Appl. Math. Comput. 2011 36(1) 1--10.
\bibitem[Fan and Jiang(2011)]{ref55}Fan, S.J., Jiang, L., 2011. Existence and uniqueness result for a backward stochastic differential equation whose generator is Lipschitz continuous in y and uniformly continuous in z. J. Appl. Math. Comput. 36(1), 1--10.

%\bibitem{ref19}W. Fleming, H. Soner. Controlled Markov processes and viscosity solutions. Springer Science and Business Media, 2006.
%
%\bibitem{ref27}C. Frei, G. Dos Reis. Quadratic FBSDE with generalized Burgers' type nonlinearities, perturbations and large deviations, Stoch. Dynam., 13 (2013) 1250015.
%
%\bibitem{ref2}M. Freidlin, A. Wentzell. Random Perturbations of Dynamical Systems, Springer--Verlag Berlin Heidelberg, 2012.

%\bibitem{ref63}Y. Hu, P. Imkeller, M. Muller. Utility maximization in incomplete markets, Ann. Appl. Probab. 15(3) (2005) 1691--1712.


%\bibitem{ref32}I. A. Kachanova, S. Y. Makhno. Large deviations for the backward stochastic differential equations, Random Oper. Stoch. Equ. (2012) 20 197--208.

%\bibitem{ref41}N. Kazamaki. Continuous exponential martingales and BMO, Springer, 2006.
\bibitem[Kazamaki(2006)]{ref41}Kazamaki, N., 2006. Continuous exponential martingales and BMO. Springer.

%\bibitem{ref25}M. Kobylanski. Backward stochastic differential equations and partial differential equations with quadratic growth, Ann. Probab. (2000) 558--602.
%\bibitem{ref25}Kobylanski, M., 2000. Backward stochastic differential equations and partial differential equations with quadratic growth. Ann. Probab. 28 (2), 558--602.
\bibitem[Kobylanski(2000)]{ref25}Kobylanski, M., 2000. Backward stochastic differential equations and partial differential equations with quadratic growth. Ann. Probab. 28 (2), 558--602.

%\bibitem{ref3}H. Kunita. Stochastic flows and stochastic differential equations, Cambridge University Press, 1997.

%\bibitem{ref30}J. Lepeltier, J. San Martin. Backward stochastic differential equations with continuous coefficient, Stat. Probab. Lett. (1997) 32(4) 425--430.
\bibitem[Lepeltier and San Martin(1997)]{ref30}Lepeltier. J.P., San Martin. J. 1997. Backward stochastic differential equations with continuous coefficient. Statist. Probab. Lett. 32 (4), 425--430.

%\bibitem{ref52}J. Lepeltier, J. San Martin. Existence for BSDE with superlinear-quadratic coefficient, Stoch. Stoch. Rep. (1998) 63(3-4) 227--240.

%\bibitem{ref17}J. Ma, T. Zajic. Rough Asymptotics of Forward--Backward Stochastic Differential Equations. Control of Distributed Parameter and Stochastic Systems, Springer, Boston, MA, 1999, 239--246.

%\bibitem{ref64}M. Mania, M. Schweizer. Dynamic exponential utility indifference valuation, Ann. Appl. Probab. 15(3) (2005) 2113--2143.

%\bibitem{ref28}D. Nualart. The Malliavin calculus and related topics, Berlin, Springer, 2006.
\bibitem[Nualart(2006)]{ref28}Nualart, D., 2006. The Malliavin calculus and related topics. Springer.

%\bibitem{ref5}E. Pardoux, S. Peng. Adapted solution of a backward stochastic differential equation, Syst. Control Lett. (1990) 14(1) 55--61.
\bibitem[Pardoux and Peng(1990)]{ref5}Pardoux, E., Peng, S.G., 1990. Adapted solution of backward stochastic differential equation. Systems Control Lett. 14 (1), 55--61.

%\bibitem{ref4}E. Pardoux, S. Peng. Backward stochastic differential equations and quasilinear parabolic partial differential equations, in: B.L. Rozuvskii, R.B. Sowers (Eds.), Stochastic Partial Differential Equations and Their Applications, in: Lect. Notes Control Inf. Sic., vol.176, Springer, Berlin, Heidelberg, 1992, 200--217.

%\bibitem{ref13}H. Pham. Some applications and methods of large deviations in finance and insurance, Paris--Princeton Lectures on Mathematical Finance . Springer, Berlin, Heidelberg, 2007, 191--244.
%
%\bibitem{ref15}S. Rainero. Un principe de grandes d\'eviations pour une \'equation diff\'erentielle stochastique progressive r\'etrograde, C. R. Math. Acad. Sci. Paris 343 (2006) 141--144.

%\bibitem{ref81}A. Richou. Markovian quadratic and superquadratic BSDEs with an unbounded terminal condition, Stoch. Process. Appl. (2012) 122(9) 3173--3208.
\bibitem[Richou(2012)]{ref81}Richou, A., 2012. Markovian quadratic and superquadratic BSDEs with an unbounded terminal condition. Stoch. Process. Appl. 122(9), 3173--3208.

%\bibitem{ref65}R. Rouge, N. El Karoui. Pricing via utility maximization and entropy, Math. Finance (2000) 10(2) 259--276.

%\bibitem{ref60}R. Tevzadze. Solvability of backward stochastic differential equations with quadratic growth, Stoch. Proces. Appl. (2008) 118(3) 503--515.

%\bibitem{ref21}S. Varadhan. Large Deviations and Applications, Philadelphia, Society for Industrial and Applied Mathematics, 1984.

%\bibitem{ref26} J. Zhang. Backward Stochastic Differential Equations: From Linear to Fully Nonlinear Theory, volume 86, Springer, New York, 2017.
\bibitem[Zhang(2017)]{ref26}Zhang, J., 2017. Backward Stochastic Differential Equations: From Linear to Fully Nonlinear Theory. volume 86, Springer.



\end{thebibliography}
\end{document}